\documentclass[12pt]{article}
\usepackage{amssymb,amsmath,textcomp,tabularx,float,amsthm}
\usepackage{cases}
\usepackage{textgreek}
\usepackage{pstricks}
\usepackage{tikz}
\usepackage{pgfplots}
\usepackage{epsfig}
\usepackage{ae,aecompl}
\newtheorem{theorem}{Theorem}[section]

\theoremstyle{definition}

\newtheorem{remark}{Remark}[section]

\newcommand{\abs}[1]{\left\vert#1\right\vert}
\newcommand{\RR}{\mathbb{R}}

\newcommand{\norm}[1]{\left\Vert#1\right\Vert}

\newtheorem{proposition}{Proposition}[section]

\begin{document}
\title{A $k$-Hessian equation with a power nonlinearity source and self-similarity}

\author{Justino S\'{a}nchez}
\date{}
\maketitle
\begin{center}
Departamento de Matem\'{a}ticas, Universidad de La Serena\\
 Avenida Cisternas 1200, La Serena, Chile.
\\email: jsanchez@userena.cl
\end{center}

\begin{abstract}
We study existence and uniqueness of spherically symmetric solutions of 
\begin{equation*}
S_k(D^2v)+\beta \xi\cdot\nabla v+\alpha v+\abs{v}^{q-1}v=0\;\; \mbox{in}\;\; \mathbb{R}^n,
\end{equation*}
where $\alpha,\beta$ are real parameters, $n>2,\, q>k\geq 1$ and $S_k(D^2v)$ stands for the $k$-Hessian operator of $v$. Our results are based mainly on the analysis of an associated dynamical system and energy methods. We derive some properties of the solutions of the above equation for different ranges of the parameters $\alpha$ and $\beta$. In particular, we describe with precision its asymptotic behavior at infinity. Further, according to the position of $q$ with respect to the first critical exponent $\frac{(n+2)k}{n}$ and the Tso critical exponent $\frac{(n+2)k}{n-2k}$ we study the existence of three classes of solutions: crossing, slow decay or fast decay solutions. In particular, if $k>1$ all the fast decay solutions have a compact support in $\RR^n$. The results also apply to construct self-similar solutions of type I to a related nonlinear evolution equation. These are self-similar functions of the form $u(t,x)=t^{-\alpha}v(xt^{-\beta})$ with suitable $\alpha$ and $\beta$. 
\end{abstract}



\section{Introduction and main results} 
Our main goal is to determine existence, uniqueness and main properties of the radial solutions of the governing equation
\begin{equation}\label{eq:maineqv}
S_k(D^2v)+\beta\xi\cdot\nabla v+\alpha v+\abs{v}^{q-1}v=0\;\; \mbox{in}\;\; \mathbb{R}^n,
\end{equation}
where $n>2,\, q>k,\,\alpha,\,\beta$ are real parameters, $\xi\cdot\nabla=\sum_{i=1}^{n}\xi_i\partial/\partial\xi_i$ and $S_k(D^2v)$ is the $k$-Hessian operator of $v$. This operator is defined as the $k$-th elementary symmetric polynomial of eigenvalues of the Hessian matrix. Equivalently, $S_k(D^2v)$ is the sum of the $k\times k$ principal minors of the Hessian matrix $D^2v$. This class of operators contains in particular the Laplace operator $S_1(D^2 u)=\Delta u$ and the Monge-Amp\`{e}re operator $S_n(D^2 u)=\mbox{det}\,D^2 u$. When $k\geq 2$, the $k$-Hessian operators are fully nonlinear. They are elliptic restricted to a subset of twice continuously differentiable functions, which are called $k$-{\it admissible} or $k$-{\it convex} functions. Equations involving $k$-Hessian operators has many applications in geometry, optimization theory and in other related fields. For more information about these operators, see for instance \cite{Wang94, Wang09}.

We point out that an interesting case of equation \eqref{eq:maineqv} was studied by Haraux and Weissler \cite{HaWe82} when $k=1,\; \beta=1/2,\;\alpha=1/(\gamma-1)$ and $q=\gamma$. In this case, \eqref{eq:maineqv} reduces to
\[
\Delta v+\frac{1}{2}\,\xi\cdot\nabla v+\frac{1}{\gamma-1}\,v+\abs{v}^{\gamma-1}v=0\;\; \mbox{in}\;\; \mathbb{R}^n.
\] 
Assuming radial symmetry, this equation can be written as
\begin{equation}\label{eq:Lapla}
v''(r)+\left(\frac{n-1}{r}+\frac{r}{2}\right)v'(r)+\frac{v(r)}{\gamma-1}+\abs{v(r)}^{\gamma-1}v(r)=0,\;\; r=\abs{\xi}>0,
\end{equation}
which has been analyzed extensively in \cite{PeTeWe86, Qi98, Weissler85, WeisslerFB85}. 

In particular, in \cite{HaWe82} among others results, the following results were obtained: Let $\gamma>1$ and $n\geq 1$. For all real number $v_0$ there is a unique function $v$ of class $C^2$ for $r\geq 0$ with $v(0)=v_0$ and $v'(0)=0$, which satisfies \eqref{eq:Lapla} for all $r>0$. Moreover, $\lim_{r\rightarrow\infty}r^{2/(\gamma-1)}v(r)=L$ always exists and is finite. Furthermore:
\begin{itemize}
\item [$(i)$] If $n(\gamma-1)/2\leq 1$, then $v(r)=0$ for at least one value of $r>0$.

\item [$(ii)$] If $1<n(\gamma-1)/2<\gamma+1$, then for sufficiently small $v_0>0,\, v(r)>0$ for all $r>0$, and $L>0$; and for at least some $v_0>0$ there is an $r>0$ such that $v(r)=0$. Let $v_{\inf}=\inf \{v_0>0: v(r)=0\; \mbox{for some } r>0\}$. If $v_0=v_{\inf}$, then $v(r)>0$ for all $r>0$, and $L=0$.

\item [$(iii)$] If $\gamma+1\leq n(\gamma-1)/2$ and $v_0>0$, then $v(r)>0$ for all $r>0$, and $L>0$.
\end{itemize}
As an application of these results, the authors in \cite{HaWe82} use the function $v$ to construct solutions of the parabolic partial differential equation $u_t=\Delta u+\abs{u}^{\gamma-1}u$ of the special form $u(t,x)=t^{-1/(\gamma-1)}v(t^{-1/2}\abs{x})$. 

An extension of the results of Haraux-Weissler to a more general equation involving the $p$-Laplacian operator was done by M. F. Bidaut-V\'{e}ron \cite{BV06}. Several results were obtained, which include a complete description of the behavior of the solutions near infinity, the existence of slow decay solutions, fast decay solutions (positive or changing sign), solutions with compact support and oscillatory solutions are also studied. Thus some results obtained in \cite{HaWe82}, \cite{Qi98}, \cite{Weissler85} were generalized and improved.

As we will see the governing equation \eqref{eq:maineqv} is also related to the study of special solutions of the evolution equation with a source term of power type
\begin{equation}\label{eq:k-evol}
u_t=S_{k}(D^2 u)+\abs{u}^{q-1}u\;\; \mbox{in}\;\; (0,T)\times\mathbb{R}^n.
\end{equation}
Here $u = u(t, x)\in\mathbb{R},\; 0<T\leq\infty$ is the existence time of $u$ and we sometimes write $u(t)$ for the spatial function $u(t,\cdot)$. We note that usually, when dealing with evolution equations with $k$-Hessian operators, some natural restrictions are often imposed on the solution $u$, such as convexity in $x$, and monotonicity in $t$. Such restrictions are typical of Cauchy problems on bounded domains with homogeneous Dirichlet boundary conditions. However, since we are interested in special solutions of \eqref{eq:k-evol} that are defined on the entire space, we do not impose such restrictions in this work. 

Note that, for $\lambda>0$, the scaling group
\[
t\rightarrow \lambda^{2k(q-1)/(q-k)}t,\;  x\rightarrow \lambda x,\; u\rightarrow \lambda^{2k/(q-k)}u
\]
preserves solutions of \eqref{eq:k-evol}. Solutions which are invariant under rescaling are called {\it self-similar}.

The configuration of the parameters $\alpha$ and $\beta$ in \eqref{eq:maineqv} permit us to construct self-similar solutions of equation \eqref{eq:k-evol}, for example, self-similar solutions of type I, i.e of the form $u(t,x)=t^{-\alpha}v(xt^{-\beta})$. For these solutions $T=\infty$ in \eqref{eq:k-evol}, so they are global (in time) solutions. The exponents $\alpha$ and $\beta$ satisfies the relation $(k-1)\alpha+2k\beta=1$.
We take $\alpha=\alpha_0=1/(q-1)$ and $\beta=\beta_0=(q-k)/2k(q-1)=((q-k)/2k)\alpha_0$, thus $u(t,x)=t^{-1/(q-1)}v(xt^{-(q-k)/2k(q-1)})$.
The function $v(\xi)=v(x)=u(1,x)$ is called the {\it profile} of the solution $u$ and is a solution of \eqref{eq:maineqv} with $q>k\geq 1$. 

A particular case of equation \eqref{eq:k-evol} is related to one of the models studied by Budd and Galaktionov in \cite{BuGa13}, namely, the Monge-Amp\`{e}re flow generated by the fully nonlinear PDE
\begin{equation}\label{eq.M-Aflow}
u_t=(-1)^{d-1}\abs{D^2u}+\abs{u}^{p-1}u\;\; \mbox{in}\;\; (0,\infty)\times\RR^d,
\end{equation}
where $\abs{D^2u}\equiv \mbox{det}\,D^2 u$ and $p>1$ is a given constant. The scalar factor in front of the main operator ensures local well-posedness (local parabolicity) of the partial differential equation. We point out that many of the results in \cite{BuGa13} are based on the existence of self-similar solutions. Thus the study of such solutions is relevant. In fact, the computational and analytical evidence shows that the phenomenon of finite time blow up that exhibited equation \eqref{eq.M-Aflow} is described by self-similar solutions. Notice that for odd dimensions $d$, equation \eqref{eq.M-Aflow} coincides with equation \eqref{eq:k-evol} when $d=k=n$ and $p=q$. Nonlinear equations involving Monge-Amp\`{e}re operators arise in many problems related to, for example, optimal transport and geometric flows, image registration and the evolution of vorticity in meteorological systems. The reader is suggested to consult \cite{BuGa13} for an extensive list of references on these applications. 

Very recently \cite{Sanchez25}, we studied the equation \eqref{eq:maineqv} without the nonlinear term $\abs{u}^{q-1}u$. As an application, we constructed several explicit families of solutions in self-similar forms of a homogeneous version of equation \eqref{eq:k-evol}. Some of these families include self-similar solutions of type I. For the reader's convenience, we recall the explicit form of these solutions:
\begin{equation}\label{eq:k-Baren}
U_C(t,x)=t^{-\alpha}\left(C-\gamma\abs{x}^2 t^{-2\beta}\right)_{+}^\frac{k}{k-1},
\end{equation}
where $(s)_+=\max\{s,0\}, C>0$ is a free constant, and $\alpha, \beta$ and $\gamma$ have explicit values, namely
\[
\alpha=\frac{n}{n(k-1)+2k},\;\;\;\; \beta=\frac{1}{n(k-1)+2k},\;\;\;\; \gamma=\frac{k-1}{2k}\left(\frac{\beta}{c_{n,k}}\right)^\frac{1}{k},\;\;\;\; c_{n,k}=\frac{1}{n}\binom{n}{k}.
\]
The full family of solutions $U_C$ ($k>1$) appears for the first time in \cite{Sanchez20}. Note that these solutions exhibit several common properties which have been previously observed in self-similar solutions of PDEs with classic reaction-diffusion, e.g. the porous medium equation $u_t=\Delta (u^m)\; (m>1)$ and the $p$-Laplacian evolution equation $u_t=\mbox{div}(\abs{\nabla u}^{p-2}\nabla u)$ when $p>2$. Some of these properties, such as compact support in space for every fixed time, conservation of mass (i.e., $\int_{\RR^n}\abs{u}(t,x)dx$ is a constant function of time) and  finite speed of propagation, are reminiscent of the properties of the famous family of Barenblatt solutions to the porous media equation. We refer the reader to \cite{Vazquez07} and related references therein for further details.

In this paper we study the full equation \eqref{eq:maineqv} which includes a polynomial type source term that was not previously considered in \cite{Sanchez20, Sanchez25} and, up to our knowledge, it is the first time that self-similar solutions of type I are studied for this model equation. To do that, we follow the approach of \cite{BV06} by means of energy functions, scaling techniques and analysis of systems of ODEs. A major technical difficulty, however, is obtaining adequate energy functions for equation \eqref{eq:rad_kappa_0}. To save the method, we need to introduce a change of variables that reduces the problem to the study of an equivalent equation. In this new scenario many of the arguments used in \cite{BV06} apply very easily to the new equation. Nevertheless, some energy functions do not work properly. Consequently, further elaboration is necessary including some modifications and the study of both equations in parallel. For this, we choose suitable parameters to define new energy functions. 

To establish our results, we need the values:
\[
q_c(k):=\frac{(n+2)k}{n}\;\; \mbox{and }\; q^\ast(k):=\begin{cases}
\frac{(n+2)k}{n-2k},\; 1\leq k<\frac{n}{2}\\
\infty,\; \frac{n}{2}\leq k\leq n.
\end{cases}
\]
The first value coincides when $k=1$ with the {\it Fujita exponent}, $q_F=1+\frac{2}{n}$, which was introduced in the pioneering paper of H. Fujita \cite{Fujita66}. In the semilinear case, this critical exponent appear in the study of nonnegative solutions of \eqref{eq:k-evol} subject to an initial datum $u(0)$. This value separates the range of exponents $q$ for which a solution blow up or exist for all times in some appropriate sense, depending on a small or large datum. We give some further comments on the value $q_c(k)$. A well known heuristic procedure that uses scaling arguments to predict the correct value of $q_F$ for a large number of nonlinear equations is to balance the decay rates of solutions of specific equations. In our setting this correspond to equalize the rate of decay $t^{-\frac{n}{n(k-1)+2k}}$ of the self-similar solution \eqref{eq:k-Baren} and the rate of decay of the solution of $u_t=u^q$ given by $C(T-t)^{-\frac{1}{q-1}}$; in other words we have the relation
\[
\frac{n}{n(k-1)+2k}=\frac{1}{q-1}\Longleftrightarrow q=\frac{(n+2)k}{n}=q_c(k).
\]

It should be noted that while the above procedure gives correct results in predicting Fujita critical exponents for a large number of nonlinear equations, it is not the rule. See the example given by Cazenave et al. \cite{CaDiWe08} where a finite Fujita critical exponent is not given by scaling. Nevertheless, it is interesting to note that the exponent $q_c(k)$ is the same as the Fujita critical exponent for a Monge-Amp\`{e}re equation with radial structure, the equation \eqref{eq.M-Aflow} above (see \cite[Section 2]{BuGa13} for more details about the model). A natural question arise: there is others $k$-Hessian equations for which $q_c(k)$ is the Fujita critical exponent?. Now, as is well known, the second value, $q^*(k)$, corresponds to the critical exponent for the $k$-Hessian operator on a ball introduced by Tso in \cite{Tso90}.

On the other hand, for radial solutions $v(\xi)=v(r),\,r=\abs{\xi}>0$, equation \eqref{eq:maineqv} translates into the nonlinear differential equation 
\begin{equation}\label{eq:radnde}
c_{n,k}r^{1-n}(r^{n-k}(v')^k)'+\beta rv'+\alpha v+\abs{v}^{q-1}v=0\;\; \mbox{in }\; (0,\infty),
\end{equation}
where $c_{n,k}=\binom{n}{k}/n$ is a binomial coefficient. 

We restrict our attention to self-similar solutions of type I slightly modified. Namely
\[
u(t,x)=(\epsilon t)^{-\alpha_0}v((\epsilon t)^{-\beta_0}x).
\]
The dilation $\epsilon t$ of $t$ by the factor $\epsilon$ is introduced so that the two parameters $\alpha_0$ and $\beta_0$ are linked by a single parameter $\kappa_0$. In fact for radial profiles $v$, from equation \eqref{eq:radnde}, we have 
\begin{equation}\label{eq:rad_kappa_0}
(r^{n-k}(v')^k)'+r^{n-1}(rv'+\kappa_0 v+c_{n,k}^{-1}\abs{v}^{q-1}v)=0,
\end{equation}
where $r=(\epsilon t)^{-\beta_0}\abs{x},\;\epsilon=c_{n,k}/\beta_0$ and $\kappa_0=\alpha_0/\beta_0$. Due to the condition on these two parameters given by $(k-1)\alpha_0+2k\beta_0=1$, there is a bijection between $\kappa_0$ and the couple $(\alpha_0,\beta_0)$, so that the parameter values are easily recovered from $\kappa_0$.

The main consequence of the analysis of the solutions of equation \eqref{eq:maineqv} is a result concerning type I radially symmetric self-similar solutions to equation \eqref{eq:k-evol} stated in the following theorem. 
\begin{theorem}\label{beha_type_I_sol}
Suppose that $k$ is an odd integer greater or equal to 1. Let $q>k$. 
\begin{itemize}
\item [$(a)$] For any $\gamma>0$, there exists a unique self-similar solution of \eqref{eq:k-evol} of the form
\begin{equation}\label{eq:type_I}
u(t,x)=(c_{n,k}\beta_0^{-1}t)^{-\frac{1}{q-1}}v((c_{n,k}\beta_0^{-1}t)^{-\beta_0}\abs{x})
\end{equation}
such that $v\in C^2((0,\infty))\cap C^1([0,\infty)),\; v(0)=\gamma$ and $v'(0)=0$. Any profile $v$ of a solution of this form satisfies $\lim_{\abs{\xi}\rightarrow\infty}\abs{\xi}^{\kappa_0}v(\abs{\xi})=L\in\mathbb{R}$.

\item [$(b)$] If $q_c(k)<q$, there exists positive solutions with $L>0$, called slow decay solutions (slow solutions, for short).

\item [$(c)$] If $q_c(k)<q<q^\ast(k)$, there exists a nonnegative solution $v\not\equiv 0$ such that $L=0$, called fast decay solution (fast solution, for short), and 
\[
u(t)\in L^p(\mathbb{R}^n)\; \mbox{for any }\; p\geq 1,\; \lim_{t\rightarrow 0}\norm{u(t)}_p=0\; \mbox{whenever }\; p<\frac{n}{\kappa_0},
\]
\[
\lim_{t\rightarrow 0}\sup_{\abs{x}\geq\epsilon}\abs{u(t,x)}=0\;\; \mbox{for any }\; \epsilon>0.
\]
If $k>1$, $v$ has a compact support.

\item [$(d)$] If $q\geq q^\ast(k)$, all the solutions $v\not\equiv 0$ have a constant sign and are slow solutions.

\item [$(e)$] If $q\leq q_c(k)$, all the solutions $v\not\equiv 0$ assume both positive and negative values. These solutions are called crossing solutions. There exists an infinite number of fast solutions $v$ with compact support.
\end{itemize}
\end{theorem}

 For the reader's convenience, we summarize the results of Theorem \ref{beha_type_I_sol}
in the following table:
\begin{table}[H]
\centering
\begin{tabularx}{\textwidth}{| X | X | X | X | X| X|} 
\hline
 \small $k<q\leq q_c(k)$  & \small $q_c(k)<q<q^*(k)$ & 
 \small $q\geq q^*(k)$\\
\hline
 \small All nontrivial solutions are crossing solutions. There exists an infinite number of fast decay solutions with compact support.
&There exists both slow and fast solutions. If $k>1$, the fast decay solutions have compact support.
 \small 
& All nontrivial solutions are either positive or negative and are slow decay solutions. 
\small  
\\ 
\hline
\end{tabularx}
\caption{Classification of the profile $v$ of the solution $u(t,x)$ of the form \eqref{eq:type_I} depending on the exponent $q$.}
\end{table}
 
Some remarks are in order.
\begin{remark}
The statement $(ii)$ for the semilinear equation \eqref{eq:Lapla} is in stark contrast with the case $k>1$ in statement $(c)$ of Theorem \ref{beha_type_I_sol}, where the profiles of the solutions are compactly supported in that range of exponents.
\end{remark}

\begin{remark}
Let $u$ be the self-similar solution of equation \eqref{eq:k-evol} of the form \eqref{eq:type_I}. So, in addition to the results obtained in part $(c)$ of Theorem \ref{beha_type_I_sol}, we can also obtain information about the pointwise behavior of $u(t,x)$ as $t\rightarrow 0$. In this respect, the fast decay solutions are of special interest. In fact, choose $\gamma>0$ such that $L=L(\gamma)=0$ and let
\[
u(t,x)=(c_{n,k}\beta_0^{-1}t)^{-\frac{1}{q-1}}v(\abs{\xi}),\, \xi=(c_{n,k}\beta_0^{-1}t)^{-\beta_0}x.
\]
If $x\neq 0$, it is easy to see that
\[
\lim_{t\rightarrow 0}u(t,x)=\abs{x}^{-\kappa_0}\lim_{r\rightarrow\infty}r^{\kappa_0}v(r)=\abs{x}^{-\kappa_0}L=0.
\]
Then $u(t,x)$ has initial value $u(0)=0$ pointwise a.e.
\end{remark}

\begin{remark}
Let $u$ be a crossing solution corresponding to part $(e)$ of Theorem \ref{beha_type_I_sol}. Let $p\geq 1$ and let $t>0$. If $v$ is a fast decay solution, then
\[
\norm{u(t)}_p=(c_{n,k}\beta_0^{-1}t)^{\frac{n\beta_0}{p}-\frac{1}{q-1}}\norm{v}_p,
\]
so that the norm of $u(t)$ in $L_p(\RR^n)$ is independent of $t$ if and only if $p=\frac{n(q-k)}{2k}$. In particular, $\norm{\cdot}_1$ is invariant by scaling if and only if $q=q_c(k)$.
\end{remark}

\begin{remark}
Let $u$ be a self-similar solution of \eqref{eq:k-evol} of type I whose existence is guaranteed by part $(e)$ of Theorem \ref{beha_type_I_sol}. Then for the range of exponents $q<q_c(k)$ the singularity of the solution $u$ at $(0,0)$ is stronger than the singularity of $U_C$ in \eqref{eq:k-Baren}. To see this, note for example that
\[
U_C(t,0)=C^{\frac{k}{k-1}}t^{-\alpha}
\]
while
\[
u(t,0)=(c_{n,k}\beta_0^{-1}t)^{-\frac{1}{q-1}}v(0)
\]	
and $\frac{1}{q-1}<\frac{n}{n(k-1)+2k}=\alpha$ because $q<q_c(k)$.	
\end{remark}
	
The paper is organized as follows. In Section 2 we define a change of variables which allow us to transform the radial fully nonlinear equation into an equivalent one but much easier to handle. With this new equation and the energy associated with it we show some basic properties of the solutions. A result about the possible existence of crossing solutions and positive solutions is also given in this section. In Section 3 we investigate the number of zeros of crossing solutions. Estimates on the solutions and its derivatives are given in Section 4. In Section 5 we present the main results concerning the precise asymptotic behavior near infinity of the solutions. First, for the solutions $\theta$ of the equivalent problem \eqref{eq:newgoveq}-\eqref{eq:newinicond} and then, using the transformation \eqref{eq:cv}, for the solutions $v$ of problem \eqref{eq:goveq}-\eqref{eq:inicond}. Section 6 is dedicated to the study of nonnegative solutions $\theta$ of equation \eqref{eq:newgoveq}. We prove the existence of slow decay solutions, fast decay solutions and solutions with compact support for different ranges of parameters $\gamma, \kappa, q$ and dimension $n$. Applying the results of the previous sections the proof of Theorem \ref{beha_type_I_sol} follows, by taking $\kappa=\kappa_0$. Finally, in Section 7, we give an example of the application of our general results.

\section{Change of variables and energy functions}\label{cvef}
In this section, we begin by rewriting our equation in a suitable equivalent form. Next, we introduce a change of variables that reduce the equation to an equivalent one. The new equation has associated a natural energy function. From this energy and a generalized version of this, many properties of the solutions follows easily.

To proceed with our study, in a manner similar to the approach of \cite{BV06}, we consider
a more general equation, which contain \eqref{eq:rad_kappa_0} as a particular case
\begin{equation}\label{eq:goveq} 
(r^{n-k}(v')^k)'+r^{n-1}(rv'+\kappa v+c_{n,k}^{-1}\abs{v}^{q-1}v)=0\;\; \mbox{in}\;\; (0,\infty).
\end{equation}
This equation only contains $\kappa$ as a parameter and is subject, for $\gamma\neq 0$, to the initial conditions
\begin{equation}
\label{eq:inicond}
\begin{cases}
v(0)=\gamma,\\
v'(0)=0.
\end{cases}
\end{equation}
We are interesting in global solutions $v\in C^2((0,\infty))\cap C^1([0,\infty))$ of problem \eqref{eq:goveq}-\eqref{eq:inicond}.
By symmetry, we restrict our study to the case when $\gamma>0$ since $k$ is odd. For a solution $v=v(\cdot,\gamma)$ of problem \eqref{eq:goveq}-\eqref{eq:inicond}, we define $L(\gamma)\equiv\lim_{r\rightarrow\infty}r^{\kappa}v(r)=L$. As we will see this limit always exists as a finite number.

Note that equation in \eqref{eq:goveq} can be written under the equivalent form 
\begin{equation}\label{eq:equiv}
(r^n(v+r^{-k}(v')^k))'+r^{n-1}((\kappa-n)v+c_{n,k}^{-1}\abs{v}^{q-1}v)=0\;\; \mbox{in}\;\; (0,\infty).
\end{equation}
Defining 
\[
J_n(r)=r^n(v+r^{-k}(v')^k),
\]
we see that \eqref{eq:equiv} is equivalent to
\begin{equation}\label{eq:Jotanprima}
J'_n(r)=r^{n-1}(n-\kappa-c_{n,k}^{-1}\abs{v}^{q-1})v.
\end{equation}

We also need the function
\begin{equation}\label{eq:Jotakappa}
J_{\kappa}(r)=r^\kappa(v+r^{-k}(v')^k)=r^{\kappa-n}J_n(r),
\end{equation}
which satisfies
\begin{equation}\label{eq:Jotakappaprima}
J_{\kappa}'(r)=r^{\kappa-1}((\kappa-n)r^{-k}(v')^k-c_{n,k}^{-1}\abs{v}^{q-1}v). 
\end{equation}
As will be seen the function $J_\kappa$ is used to prove that the limit $L(\gamma)$ defined above is a real number. 

As we say above it is more convenient to look at the problem \eqref{eq:goveq}-\eqref{eq:inicond} in a scaled version. For this, we introduce a change of variables that reduces the problem to an equivalent problem preserving the initial conditions, but in which the spatial dimension $n$ is change to a lower fractional dimension $\tilde{n}$. Due to the structure of the equation in the new variables, emerges naturally an energy function. The energy is decreasing along the solutions, this is a key point to give the first step in the analysis of the new equation. 

Let us proceed. Set
\begin{equation}\label{eq:cv}
v(r)=\theta (s),\; s=ar^b.
\end{equation}
A similar change of variables was introduced by Ph. Korman \cite{Korman20} in the context of radial $p$-Laplace equations.
Define the function $h(s,\theta,\theta')$ by $h(s,\theta,\theta'):=bs\theta'+\tilde{f}(\theta)$, where $\tilde{f}(\theta)=\kappa\theta+c_{n,k}^{-1}\abs{\theta}^{q-1}\theta$. The constants $a$ and $b$ in \eqref{eq:cv} will be chosen appropriately later. Clearly, $v'=abr^{b-1}\theta'$, and we have
 \[
 (s^{\frac{n-2k+kb}{b}}(\theta')^k)'+a^{-\frac{2k}{b}}b^{-(k+1)}s^{\frac{n-b}{b}}h(s,\theta,\theta')=0.
 \]
 We now choose $b$ to equalize the powers of $s$:
 \begin{equation}\label{eq:eqpo}
 \frac{n-2k+kb}{b}=\frac{n-b}{b}=:\delta\,,
 \end{equation}
 \begin{equation}\label{eq:b}
 b=\frac{2k}{k+1}.
 \end{equation}
 The common power $\delta$, defined in \eqref{eq:eqpo}, is then
 \begin{equation}\label{eq:delta}
 \delta=\frac{n(k+1)-2k}{2k}.
 \end{equation}
 Observe that $1\leq b<2$ and $\delta>0$.
 On the other hand, we choose $a=b^{-1}$ so that $a^{-\frac{2k}{b}}b^{-(k+1)}=1$. Hence, we obtain the equation
 \begin{equation}\label{eq:normaleq}
 (s^{\delta}(\theta')^k)'+s^{\delta}h(s,\theta,\theta')=0.
 \end{equation}
Note that the fractional dimension mentioned above is given by $\tilde{n}:=\delta+1=\frac{n(k+1)}{2k}$. 
  
Now we show that \eqref{eq:cv} preserves the initial conditions in \eqref{eq:inicond}. In fact, from \eqref{eq:goveq}, we have 
\[
(-v'(r))^k=\frac{1}{c_{n,k}r^{n-k}}\int_0^r \tau^{n-1}(\beta\tau v'(\tau)+f(v(\tau)))d\tau,
\]
since $k$ is an odd integer. 

Note that $r\rightarrow 0$ as $s\rightarrow 0$, since $b\geq 1$. Thus, using the L'Hospital rule, we have
\[
-\frac{d\theta}{ds}(0)=\lim_{r\downarrow 0}\frac{-v'(r)}{r^{b-1}}=\lim_{r\downarrow 0}\left[\frac{(-v'(r))^k}{(r^{b-1})^k}\right]^\frac{1}{k},
\]
and
\[
\lim_{r\downarrow 0}\frac{(-v'(r))^k}{(r^{b-1})^k}=\lim_{r\downarrow 0}\frac{\beta r^n v(r)+\int_0^r \tau^{n-1}((\alpha-n\beta)v(\tau)+\abs{v(\tau)}^{q-1}v(\tau))d\tau}{c_{n,k}r^{n-k+k(b-1)}}=0,
\]
since $2-b>0$ by \eqref{eq:b}.

Hence, taking into account \eqref{eq:normaleq}, we have proved that problem \eqref{eq:goveq}-\eqref{eq:inicond} is equivalent to 
\begin{equation}\label{eq:newgoveq}
(s^{\delta}(\theta')^k)'+s^{\delta}(bs\theta'+\kappa\theta+c_{n,k}^{-1}\abs{\theta}^{q-1}\theta)=0\;\; \mbox{in}\;\; (0,\infty),
\end{equation}
where a prime denotes differentiation with respect to $s>0$ and 
\begin{equation}
\label{eq:newinicond}
\begin{cases}
\theta(0)=\gamma,\\
\theta'(0)=0.
\end{cases}
\end{equation}

Until the end of this section we will dealt with equation \eqref{eq:newgoveq}. In the next sections, we will study in parallel both equations \eqref{eq:goveq} and \eqref{eq:newgoveq} and always we keeping in mind that through \eqref{eq:cv} these two equations are equivalent. Thus a property for $\theta$ implies a property for $v$ and vice versa, a property for $v$ implies a property for $\theta$. This may seem unorthodox, but it is effective. The local existence and uniqueness of a solution $\theta$ for problem \eqref{eq:newgoveq}-\eqref{eq:newinicond} for a fixed $\gamma$ is obtained by applying the contraction mapping principle and the Gronwall lemma, respectively. We omit the details.

Let $\theta$ be a solution of problem \eqref{eq:newgoveq}-\eqref{eq:newinicond}. We prove that $(\theta')^k$ is a $C^1$ function near $s=0$. Integrating the equation in \eqref{eq:newgoveq} from 0 to $s$, we obtain
\[
(\theta')^k(s)=-bs\theta(s)-s^{-\delta}\int_0^s \tau^{\delta}\theta(\tau)(\kappa-n+c_{n,k}^{-1}\abs{\theta(\tau)}^{q-1})d\tau
\]
 which yields 
\[
(\theta')^k(s)=-\frac{\gamma}{\tilde{n}}s(\kappa+c_{n,k}^{-1}\gamma^{q-1}+o(1))\;\; \mbox{as }\; s\downarrow 0.
\]
Using this and the fact that $\theta$ satisfies the equation in \eqref{eq:newgoveq}, we have
\begin{equation}\label{thetaprimezero}
\lim_{s\downarrow 0}((\theta')^k)'(s)=-\frac{\gamma}{\tilde{n}}(\kappa+c_{n,k}^{-1}\gamma^{q-1}).
\end{equation}
Hence $(\theta')^k\in C^1([0,r_\gamma))$, where $r_\gamma>0$ small enough is granted by the local existence and uniqueness near zero. Now using that $k$ is odd, we can write $\theta'=((\theta')^k)^\frac{1}{k}$ and then
\[
\theta''=\frac{1}{k}(\theta')^{1-k}((\theta')^k)'
\]
 at the points $s$ such that $\theta'(s)\neq 0$.
 
Note that equation in \eqref{eq:newgoveq} can also be written under the equivalent forms
\begin{equation}\label{eq:maineq}
((\theta')^k)'+\frac{\delta}{s}(\theta')^k+bs\theta'+\kappa\theta+c_{n,k}^{-1}\abs{\theta}^{q-1}\theta=0\;\; \mbox{in}\;\; (0,\infty).
\end{equation}
\begin{equation}\label{eq:deltaeq}
(s^{\delta+1}(b\theta+s^{-1}(\theta')^k))'+s^\delta((\kappa-n)\theta+c_{n,k}^{-1}\abs{\theta}^{q-1}\theta)=0\;\; \mbox{in}\;\; (0,\infty).
\end{equation}
Defining 
\begin{equation}\label{eq:jotadelta}
J_\delta(s)=s^{\delta+1}(b\theta+s^{-1}(\theta')^k),
\end{equation}
we see that \eqref{eq:deltaeq} is equivalent to
\begin{equation}\label{eq:jotadeltaprima}
J'_\delta(s)=s^{\delta}(n-\kappa-c_{n,k}^{-1}\abs{\theta}^{q-1})\theta.
\end{equation}

Associated to equation \eqref{eq:maineq} we have a natural energy function $\mathcal{E}$ given by
\[
\mathcal{E}(s)=\frac{k}{k+1}(\theta')^{k+1}+\frac{\kappa}{2}\,\theta^2+\frac{c_{n,k}^{-1}}{q+1}\abs{\theta}^{q+1},
\]
which is nonincreasing along the solutions, because 
\[
\mathcal{E}'(s)=-\delta s^{-1}(\theta')^{k+1}-bs(\theta')^2\leq 0.
\]
As in \cite{BV06}, we introduce a Pohozaev-type function with parameters $\lambda\geq 0$ and $\sigma, \mu\in\mathbb{R}$:
\begin{equation}\label{eq:VPoho}
\mathcal{V}_{\lambda,\sigma,\mu}(s)=s^\lambda\left(\frac{k}{k+1}(\theta')^{k+1}+\frac{c_{n,k}^{-1}}{q+1}\abs{\theta}^{q+1}+\mu\frac{\theta^2}{2}+\sigma s^{-1}(\theta')^k\theta\right).
\end{equation}
After computation we find
\begin{eqnarray*}
s^{1-\lambda}\mathcal{V}'_{\lambda,\sigma,\mu}(s)&=&-\left(\delta-\sigma-\frac{k\lambda}{k+1}\right)(\theta')^{k+1}-c_{n,k}^{-1}\left(\sigma-\frac{\lambda}{q+1}\right)
\abs{\theta}^{q+1}\\
&+&\sigma (\lambda-\delta-1)s^{-1}\theta(\theta')^k-b\left(s\theta'+\frac{b\sigma-\mu+\kappa}{2b}\theta\right)^2\\
&-&b^{-1}\left(b\kappa\sigma-\frac{b\lambda\mu}{2}
-\frac{(b\sigma-\mu+\kappa)^2}{4}\right)\theta^2.
\end{eqnarray*}
We note that, by construction, $\mathcal{E}=\mathcal{V}_{0,0,\kappa}$. So from \eqref{eq:VPoho} we recover the energy $\mathcal{E}$. 

For given $d\in\mathbb{R}$, consider the Emden-Fowler transformation
\begin{equation}\label{eq:cvlog}
\theta(s)=s^{-d}y_d(\tau),\;\; \tau=\ln s.
\end{equation}
Set $\eta=\frac{\delta}{k}-1$. The following equation is obtained for points $\tau$ such that $\theta'(s)\neq 0$.
\begin{equation}\label{eq:cveqn}
\begin{split}
&y_d''+(\eta-2d)y_d'-d(\eta-d)y_d\\
&+k^{-1}e^{((k-1)d+k+1)\tau}(y_d'-dy_d)^{1-k}(by_d'-(bd-\kappa)y_d+c_{n,k}^{-1}e^{-d(q-1)\tau}\abs{y_d}^{q-1}y_d)=0.
\end{split}
\end{equation}

Setting 
\begin{equation}\label{eq:cvsys}
Y_d(\tau)=-s^{(d+1)k}(\theta')^k,
\end{equation}
we can write \eqref{eq:cveqn} as a system
\begin{equation}\label{eq:system}
\begin{split}
&y_d'=dy_d-\abs{Y_d}^{\frac{1-k}{k}}Y_d,\\
&Y_d'=k(d-\eta)Y_d+e^{((k-1)d+k+1)\tau}(\kappa y_d+c_{n,k}^{-1}e^{-d(q-1)\tau}\abs{y_d}^{q-1}y_d-b\abs{Y_d}^{\frac{1-k}{k}}Y_d).
\end{split}
\end{equation}

The following result concern the boundedness of $\theta$ and $\theta'$ with explicit bounds.
 \begin{proposition}\label{tetateta'bound}
 For any $\gamma>0$, there exists a unique solution $\theta$ of problem \eqref{eq:newgoveq}-\eqref{eq:newinicond} in $[0,\infty)$, which satisfies
 \[
 \abs{\theta(s)}\leq\gamma,\;\; \forall s\geq 0
 \]
and
\[
 \abs{\theta'(s)}\leq\left(\frac{k+1}{k}\right)^\frac{1}{k+1}\left(\frac{\kappa}{2}\gamma^2+\frac{c_{n,k}^{-1}}{q+1}\gamma^{q+1}\right)^\frac{1}{k+1},\;\; \forall s\geq 0.
 \]
 \end{proposition}
 \begin{proof}
We have
\[
\mathcal{E}(s)=\frac{k}{k+1}(\theta')^{k+1}+\frac{\kappa}{2}\,\theta^2+\frac{c_{n,k}^{-1}}{q+1}\abs{\theta}^{q+1}\leq \mathcal{E}(0)=\frac{\kappa}{2}\,\gamma^2+\frac{c_{n,k}^{-1}}{q+1}\,\gamma^{q+1},
\]
hence $\theta$ and $\theta'$ stay bounded. In particular, $\theta$ exists for all $s\geq 0$. Note that if $\theta(s_1)=\theta'(s_1)=0$, then $\theta\equiv 0$ on $[s_1,\infty)$ because $\mathcal{E}$ is nonincreasing.
\end{proof}

We  have the following result on the behavior of the solutions at infinity.
\begin{proposition}\label{uu'vanish}
Let $\theta$ be any solution of equation in \eqref{eq:newgoveq}. Then  
\begin{equation}\label{tetateta'zero}
\lim_{s\rightarrow\infty}\theta(s)=0,\;\;\;\:\lim_{s\rightarrow\infty}\theta'(s)=0.
\end{equation}
If $\theta>0$ for large $s$, then $\theta'<0$ for large $s$.
\end{proposition}
\begin{proof}
Let $\theta$ be any solution on $[s_0,\infty)$ with $s_0>0$. Since the function $\mathcal{E}$ is non-increasing, $\theta$ and $\theta'$ are bounded. Then $\mathcal{E}$ has a nonnegative finite limit, say $l$. Consider the function $\mathcal{V}=\mathcal{V}_{\lambda,\sigma,\mu}$ defined in \eqref{eq:VPoho} with $\lambda=0,\, \sigma=\frac{\delta}{2}$ and $\mu=\kappa+b\sigma$. It is bounded near infinity and satisfies
\begin{eqnarray*}
-s\mathcal{V}'(s)&=&\frac{\delta}{2}\left((\theta')^{k+1}+\kappa\,\theta^2+c_{n,k}^{-1}\abs{\theta}^{q+1}+(\delta+1)s^{-1}\theta(\theta')^k+\frac{2b}{\delta}s^2\theta'^2\right)\\
&\geq&\frac{\delta}{2}\,\mathcal{E}(s)+o(1)\geq \frac{\delta}{2}\,l+o(1)\;\; \mbox{as }\; s\rightarrow\infty.
\end{eqnarray*}
If $l>0$, then $-\mathcal{V}'$ is not integrable, which is impossible. Thus, $l=0$, and \eqref{tetateta'zero} holds. Moreover, at each critical point $s$ such that $\theta(s)>0$, from \eqref{eq:maineq}, we have
\[
((\theta')^k)'(s)=-(\kappa+c_{n,k}^{-1}\theta(s)^{q-1})\theta(s).
\]
So, $s$ is a maximum. If $\theta(s)>0$ for large $s$, then from \eqref{tetateta'zero} necessarily $\theta'<0$ for large $s$.
\end{proof}
The final result in this section refers to the possible existence of crossing solutions. In such a case, the zeros of these solutions can be isolated or not. The existence of positive solutions of problem \eqref{eq:newgoveq}-\eqref{eq:newinicond} for sufficiently small values $\gamma$ and under a suitable condition on the parameter $\kappa$ is also obtained.
\begin{proposition}\label{zeros}
\begin{itemize}
\item [$(i)$] Suppose $\kappa<n$. Let $\underline{\gamma}=(c_{n,k}(n-\kappa))^{1/(q-1)}$. Then for any $\gamma\in (0,\underline{\gamma}],\, \theta>0$ on $[0,\infty)$.
\item [$(ii)$] Suppose $n\leq\kappa$. Then for any $\gamma>0,\, \theta(s)$ has at least one isolated zero.
\item[$(iii)$] For any $m>0$, any solution $\theta$ of equation \eqref{eq:maineq} has a finite number of isolated zeros in $[m,\infty)$, or $\theta\equiv 0$ in $[m,\infty)$.
\end{itemize}
\end{proposition}
\begin{proof}
\begin{itemize}
\item [$(i)$] Let $\gamma\in (0,\underline{\gamma}]$. Suppose that there exists a first $s_1>0$ such that $\theta(s_1)=0$, then $\theta'(s_1)\leq 0$. We now consider $J_\delta$ defined by \eqref{eq:jotadelta}. Then from \eqref{eq:jotadeltaprima}, we have that $J'_\delta(s)\geq 0$ on $[0,s_1)$, since $0<\theta(s)\leq\gamma$. Moreover, $J_\delta(0)=0$ and $J_\delta(s_1)=s_1^{\delta}(\theta')^k(s_1)\leq 0$. Hence $J'_\delta\equiv 0$ on $[0,s_1]$, and $\theta\equiv\underline{\gamma}$, a contradiction with \eqref{eq:maineq}.

\item [$(ii)$] Suppose on the contrary that for some $\gamma>0,\,\theta(s)>0$ on $[0,\infty)$. Note that $J'_\delta(s)<0$ on $(0,\infty)$, since $n\leq\kappa$. As $J_\delta(0)=0$, necessarily $J_\delta(s)\leq 0$. Then integrating the inequality $b\theta+s^{-1}(\theta')^k\leq 0$ on $(0,\infty)$ we deduce that $s\longmapsto s^\frac{k+1}{k}+\frac{k+1}{(k-1)b^{\frac{1}{k}}}\theta^\frac{k-1}{k}$ is nonincreasing, which is impossible. Thus $\theta$ has a first zero $s_1$, which satisfies $J'_\delta(s_1)<0$ on $(0,s_1)$ and $J_\delta(s_1)<0$. We conclude that $\theta'(s_1)<0$ and finally that $s_1$ is isolated.

\item[$(iii)$] Suppose that $\theta\not\equiv 0$ in $[m,\infty)$. Consider the set $\mathcal{Z}$ of all isolated zeros of $\theta$ in $[m,\infty)$. Since $(\theta(m),\theta'(m))\neq (0,0),\, m$ is not an accumulation point of $\mathcal{Z}$. Let $\nu_1<\nu_2$ be two consecutive zeros, such that $\nu_1$ is isolated and $\abs{\theta}>0$ on $(\nu_1,\nu_2)$. We perform the change of variables \eqref{eq:cvlog}, where $d>0$ will be chosen later. At each point $\tau$ such that $y_d'(\tau)=0$ and $y_d(\tau)\neq 0$, we infer
\[
ky_d''=y_d(kd(\eta-d)+e^{((k-1)d+k+1)\tau}(-dy_d)^{1-k}(bd-\kappa-c_{n,k}^{-1}e^{-d(q-1)\tau}\abs{y_d}^{q-1})).
\]
If $\tau\in (e^{\nu_1},e^{\nu_2})$ is a maximum point of $\abs{y_d}$, we deduce that
\[
e^{((k-1)d+k+1)\tau}(-dy_d(\tau))^{1-k}(bd-\kappa-c_{n,k}^{-1}e^{-d(q-1)\tau}\abs{y_d(\tau)}^{q-1})\leq kd(d-\eta).
\]
Taking $\nu=e^\tau\in (\nu_1,\nu_2)$, it read as 
\begin{equation}\label{eq:bound}
\nu^{k+1}\abs{\theta(\nu)}^{1-k}(bd-\kappa-c_{n,k}^{-1}\abs{\theta}^{q-1}(\nu))\leq kd^k(d-\eta).
\end{equation}
Now fix $d$ such that $d>b^{-1}\kappa$. Since $\lim_{s\rightarrow\infty}\theta(s)=0$, the coefficient of $\nu^{k+1}$ in the left-hand side of \eqref{eq:bound} tends to $\infty$ as $\nu\rightarrow\infty$. It follows that $\nu$ is bounded and also $\nu_1$. Hence $\mathcal{Z}$ is bounded. Suppose that $\mathcal{Z}$ is infinite, then there exists a sequence of zeros $(s_n)$, converging to some $\bar{s}\in (m,\infty)$ such that $\theta(\bar{s})=\theta'(\bar{s})=0$. It follows that there exists a sequence $(\tau_n)$ of maximum points of $\abs{y_d}$ converging to $\bar{\tau}=\ln\bar{s}$. Finally, taking $\nu=\nu_n=e^{\tau_n}$ in \eqref{eq:bound} we reach at the same contradiction as before.
\end{itemize}
\end{proof}

\section{The number of zeros of crossing solutions}
To prove part $(c)$ of Theorem \ref{beha_type_I_sol}, we need to examine the limiting behavior of $v(r,\gamma)$, suitably rescaled, as $\gamma\rightarrow\infty$. As a consequence we obtain a result concerning the number of zeros of the crossing solutions. These solutions exists in the subcritical case for large enough $\gamma$.
\begin{proposition}\label{vzeros}
Suppose $n>2k$ and $k<q<q^*(k)$. Then for every $N\in\mathbb{N}$, there exists $\overline{\gamma_N}$ such that for any $\gamma>\overline{\gamma_N},\; v(\cdot,\gamma)$ has more than $N$ isolated zeros. And for fixed $N$, the $N^\text{th}$ zero of $v(\cdot,\gamma)$ tends to 0 as $\gamma$ tends to $\infty$.
\end{proposition}
\begin{proof}
First we show that there exists $\gamma_\ast>0$ such that for any $\gamma>\gamma_\ast,\; v(\cdot,\gamma)$ cannot stay positive on $[0,\infty)$. Suppose on the contrary that there exists a sequence $(\gamma_m)$ tending to $\infty$ such that $v_m(r)=v(r,\gamma_m)\geq 0$ on $[0,\infty)$. Define
\begin{equation}\label{eq:scal}
\hat{v}_m(r)=\gamma_m^{-1}v_m(\gamma_m^{-\frac{1}{\kappa_0}}r).
\end{equation}
Then $\hat{v}_m(0)=1$ and $\hat{v}'_m(0)=0$. Further, $\hat{v}_m$ satisfies the equation
\begin{equation}\label{eq:scaleq}
(r^n(\gamma_m^{1-q}\hat{v}_m+r^{-k}(\hat{v}'_m)^k))'+r^{n-1}((\kappa-n)\gamma_m^{1-q}\hat{v}_m+c_{n,k}^{-1}\abs{\hat{v}_m}^{q-1}\hat{v}_m)=0.
\end{equation}
From \eqref{eq:cv} and Proposition \ref{tetateta'bound} applied to $v_m$ we have
\[
\hat{v}_m(r)\leq 1,\;\;\; \abs{\hat{v}'_m(r)}^{k+1}\leq\frac{k+1}{k}\left(\frac{\kappa}{2}\gamma_m^{1-q}+\frac{c_{n,k}^{-1}}{q+1}\right)\;\; \mbox{in }\;\; [0,\infty).
\]
Thus $\hat{v}_m$ and $\hat{v}'_m$ are uniformly bounded in $[0,\infty)$. From \eqref{eq:scaleq}, the derivatives of $r^n(\gamma_m^{1-q}\hat{v}_m+r^{-k}(\hat{v}'_m)^k)$ are uniformly bounded on any compact $\mathcal{K}$ of $(0,\infty)$. Moreover $\gamma_m^{1-q}\hat{v}_m$ converges uniformly to 0 in $[0,\infty)$, and by a diagonal sequence argument, $(\hat{v}'_m)^k$ converges uniformly on any $\mathcal{K}$, hence also $\hat{v}'_m$. Thus $\hat{v}_m$ converges uniformly in $C^1_{loc}(0,\infty)$ to a nonnegative function $\hat{v}\in C^1(0,\infty)$. Integrating \eqref{eq:scaleq} from 0 to $r$ yields 
\[
(\hat{v}'_m)^k(r)=-\gamma_m^{1-q}r^k\hat{v}_m+r^{k-n}\int_0^r\tau^{n-1}((n-\kappa)\gamma_m^{1-q}\hat{v}_m
-c_{n,k}^{-1}\abs{\hat{v}_m}^{q-1}\hat{v}_m)d\tau.
\]
Consequently, passing to the limit as $m\rightarrow\infty$, we obtain
\[
(\hat{v}')^k(r)=-c_{n,k}^{-1}r^{k-n}\int_0^r\tau^{n-1}\abs{\hat{v}}^{q-1}\hat{v}d\tau\;\; \mbox{in }\; (0,\infty).
\]
In particular, $\hat{v}'(r)\rightarrow 0$ as $r\rightarrow 0$, and therefore $\hat{v}$ can be extended to a function in $C^1([0,\infty))$ such that $\hat{v}(0)=1$ and $\hat{v}'(r)<0$ (recall that $k$ is odd). On the other hand, using the equivalent form
\[
((\hat{v}'_m)^k)'+\frac{n-k}{r}(\hat{v}'_m)^k+r^{k-1}(\gamma_m^{1-q}r\hat{v}'_m+\kappa\gamma_m^{1-q}\hat{v}_m+c_{n,k}^{-1}\abs{\hat{v}_m}^{q-1}\hat{v}_m=0
\]
for the equation in $\hat{v}_m$, we see that $\hat{v}''_m$ converges uniformly on any $\mathcal{K}$, whence $\hat{v}$ belong to $C^2((0,\infty))\cap C^1([0,\infty))$ and is a solution of the equation
\[
c_{n,k}r^{1-n}(r^{n-k}(\hat{v}')^k)'+\abs{\hat{v}}^{q-1}\hat{v}=0
\]
such that $\hat{v}(0)=1$ and $\hat{v}'(0)=0$. By \cite[Theorem 3.2]{ClMM98} (see Remark 3.1 on page 21 of that paper), this problem has no positive solutions in $C^2((0,\infty))\cap C^1([0,\infty))$ and hence any nontrivial solution has at least one zero, a contradiction, since then $\hat{v}$ must change sign.

The second part of the proof is the same as in \cite[Proposition 2.7, part (ii)]{BV06}. Finally, from \eqref{eq:scal} the $N^\text{th}$ zero of $v(\cdot,\gamma)$ tends to 0 as $\gamma$ tends to $\infty$.
\end{proof}

\section{Estimates of the solutions from above}
In this section we give estimates for the solutions of problems \eqref{eq:newgoveq}-\eqref{eq:newinicond} and \eqref{eq:goveq}-\eqref{eq:inicond}, respectively. Some of these estimates present the property of continuous dependence. The behavior at infinity of any solution and its derivative is also given.
These results are used in the next section to prove that for solutions $v$ of problem \eqref{eq:goveq}-\eqref{eq:inicond}, the limit $\lim\limits_{r\rightarrow\infty}r^{\kappa}v$ exists and is finite.

\begin{proposition}\label{pbehainf1}
Let $d\geq 0$.
\begin{itemize}
\item[$(i)$] Suppose that the solution $\theta$ of problem \eqref{eq:newgoveq}-\eqref{eq:newinicond} satisfies
\begin{equation}\label{eq:thetabound}
\abs{\theta(s)}\leq C_d (1+s)^{-d},
\end{equation}
on $[0,\infty)$, for some $C_d>0$. Then there exists another $D_d>0$, depending continuously on $C_d$ such that
\begin{equation}\label{eq:thetaprimabound}
\abs{\theta'(s)}\leq D_d(1+s)^{-d-1}.
\end{equation}

\item[$(ii)$] For any solution $\theta$ of \eqref{eq:newgoveq} such that $\theta(s)=O(s^{-d})$ near $\infty$, we have $\theta'(s)=O(s^{-d-1})$ near $\infty$.
\end{itemize}
\end{proposition}
\begin{proof}
\begin{itemize}
\item[$(i)$] Let $\theta$ be a nontrivial solution and let $s\geq S\geq 0$. We define
\begin{equation}\label{eq:efeS}
f_S(s)=\exp\left(\frac{b}{k}\int_S^s\tau(\theta')^{1-k}d\tau\right).
\end{equation}
By Proposition \ref{zeros}, $(iii)$, the function $\theta$ has a finite number of isolated zeros and either there exists a first $\bar{s}>0$ such that $\theta(\bar{s})=\theta'(\bar{s})=0$, or $\theta$ has no zero for large $s$, and we set $\bar{s}=\infty$. In the last case, by Proposition \ref{uu'vanish}, the set of zeros of $\theta'$ is bounded. If $\theta'(\tilde{s})=0$ for some $\tilde{s}\in(0,\bar{s})$, then $((\theta')^k)'$ has a nonzero limit $\Lambda$ at $\tilde{s}$ by \eqref{eq:maineq}, whence $\tilde{s}$ is an isolated zero of $\theta'$ and
\[
(\theta'(s))^{1-k}=\Lambda^\frac{1-k}{k}(s-\tilde{s})^{-1+\frac{1}{k}}(1+o(1))
\]
near $\tilde{s}$. Thus $\tau(\theta')^{1-k}\in L^1_{loc}(S,\infty)$, it follows that $f_S$ is absolutely continuous on $[S,\bar{s})$ if $\bar{s}=\infty$. Let $\ell=\ell(n,k,d)>0$ be a parameter such that $E=\ell-\frac{\delta}{k}>0$ and $\ell>1+d$ (recall that $\delta$ is given in \eqref{eq:delta}). By a simple computation, for almost any $s\in (S,\bar{s})$, we have
\[
(s^\ell f_S(\theta'-Es^{-1}\theta))'=-E(\ell-1)s^{\ell-2}f_S\theta-s^{\ell-1}f_S'\theta(E+b^{-1}(\kappa+c_{n,k}^{-1}\abs{\theta}^{q-1}))
\]
and then for any $s\in [S,\bar{s})$ we get
\begin{equation}\label{eq:integral}
\begin{split}
s^\ell f_S\theta'&=S^{\ell-1}(S\theta'(S)-E\theta(S))+Es^{\ell-1}f_S\theta-E(\ell-1)\int_S^s\tau^{\ell-2}f_S\theta d\tau\\
&-\int_S^s\tau^{\ell-1}f_S'\theta(E+b^{-1}(\kappa+c_{n,k}^{-1}\abs{\theta}^{q-1}))d\tau
\end{split}
\end{equation}
Assume \eqref{eq:thetabound}, take $S=0$, and divide by $f_0$. From our choice of $\ell$, and since $f_0'\geq 0$, we obtain
\[
s^\ell\abs{\theta'(s)}\leq\tilde{C}_ds^{\ell-1-d}
\]
on $[0,\tilde{s})$ and then on $[0,\infty)$, where $\tilde{C}_d=C_d\left(2E+\frac{E(\ell-1)}{\ell-1-d}+b^{-1}(\kappa+c_{n,k}^{-1}C_d^{q-1})\right)$, and $E=E(n,k,d)$. This holds in particular on $[1,\infty)$. On $[0,1]$, from Proposition \ref{tetateta'bound},
\[
 \abs{\theta'(s)}\leq\left(\frac{k+1}{k}\right)^\frac{1}{k+1}\left(\frac{\kappa}{2}C_d^2+\frac{c_{n,k}^{-1}}{q+1}C_d^{q+1}\right)^\frac{1}{k+1},
\]
and \eqref{eq:thetaprimabound} holds.

\item[$(ii)$] Let $S\geq 1$ such that $\theta$ is defined on $[S,\infty)$ and $\theta(s)\leq C_ds^{-d}$ on $[S,\infty)$. Defining $\bar{s}$ as above, dividing \eqref{eq:integral} by $f_S$ and observing from \eqref{eq:efeS} that $f_S(s)\geq 1$ and also that $S^{\ell-1-d}\leq s^{\ell-1-d}$, we deduce
\[
s^\ell \abs{\theta'(s)}\leq S^\ell\abs{\theta'(S)}+EC_dS^{\ell-1-d}+\tilde{C}_ds^{\ell-1-d}\leq (S^{1+d}\abs{\theta'(S)}+EC_d+\tilde{C}_d)s^{\ell-1-d}
\]
on $[S,\bar{s})$ and then on $[S,\infty)$, and the conclusion follows again.
\end{itemize}
\end{proof}

\begin{proposition}\label{pbehainf2}
\begin{itemize}
\item [$(i)$] For any $\mu\geq 0$, any solution of \eqref{eq:goveq} satisfies, near $\infty$,
\begin{equation}\label{eq:boundinf1}
v(r)=O(r^{-\mu})+O(r^{-\kappa})
\end{equation}

\item[$(ii)$] The solution $v=v(\cdot,\gamma)$ of problem \eqref{eq:goveq}-\eqref{eq:inicond} satisfies
\begin{equation}\label{eq:boundinf2}
\abs{v(r,\gamma)}\leq C_{\mu}(\gamma)((1+r)^{-\mu}+(1+r)^{-\kappa}),
\end{equation}
where $C_{\mu}(\gamma)$ is continuous with respect to $\gamma$ on $\mathbb{R}$.
\end{itemize}
\end{proposition}
\begin{proof}
\begin{itemize}
\item[$(i)$] Define, for $r>0$, the function 
\[
\mathcal{F}(r)=\frac{1}{2}v^2+r^{-k}(v')^kv.
\]
From \eqref{eq:goveq}, the function $\mathcal{F}$ satisfies
\begin{eqnarray*}
(r^{2\kappa}\mathcal{F}(r))'&=&r^{2\kappa-k}((v')^{k+1}+(2\kappa-n)r^{-1}(v')^kv-c_{n,k}^{-1}r^{k-1}\abs{v}^{q+1})\\
&\leq&r^{2\kappa-k}((v')^{k+1}+(2\kappa-n)r^{-1}(v')^kv).
\end{eqnarray*}
Suppose that for some $d\geq 0$ and $R>0$, $\abs{v(r)}\leq Cr^{-d}$ on $[R,\infty)$. Then by the change of variables \eqref{eq:cv} and by Proposition \ref{pbehainf1} there exists others constants $C>0$ such that $(r^{2\kappa}\mathcal{F})'\leq Cr^{2\kappa-k-(d+1)(k+1)}$ on $[R,\infty)$. Then $\mathcal{F}(r)\leq C(r^{-(k-1)-(d+1)(k+1)}+r^{-2\kappa})$ on $[R,\infty)$ if $(d+1)(k+1)+k-1\neq 2\kappa$. Moreover, $r^{-k}\abs{v'}^k\abs{v}\leq Cr^{-(k-1)-(d+1)(k+1)}$, which leads to 
\[
\abs{v(r)}\leq C(r^{-\frac{k-1+(d+1)(k+1)}{2}}+r^{-\kappa})
\]
on $[R,\infty)$. By Proposition \ref{uu'vanish}, $\theta$ is bounded on $[b^{-1}R^b,\infty)$, it follows that $v$ is bounded on $[R,\infty)$. We conclude by a recursive argument considering the sequence $(d_n)$ defined by $d_0=0,\, d_{n+1}=\frac{(d_n+1)(k+1)+k-1}{2}$. Since $k>1$, it is increasing and tends to $\infty$. After a finite number of steps, we get \eqref{eq:boundinf1} by changing slightly the sequence if it takes the value $\frac{2\kappa-k+1}{k+1}-1$.

\item[$(ii)$] By Proposition \ref{tetateta'bound}, $\abs{v(\cdot,\gamma)}=\abs{\theta(\cdot,\gamma)}\leq\gamma$. Suppose that for some $d>0, 
\abs{v(r,\gamma)}\leq C_d(\gamma)(1+r)^{-d}$ on $[0,\infty)$ and $C_d$ is continuous. Using again the change of variables \eqref{eq:cv}, we deduce that $
\abs{\theta(\tilde{s},\gamma)}\leq C_d(\gamma)(1+\tilde{s})^{-d}$ on $[0,\infty)$, where $\tilde{s}=(bs)^\frac{1}{b}$, then by Proposition \ref{pbehainf1}
\[
\abs{v(r,\gamma)}= \abs{\theta(\tilde{s},\gamma)}\leq\tilde{C}_d(\gamma)((1+\tilde{s})^{-\frac{k-1+(d+1)(k+1)}{2}}+(1+\tilde{s})^{-\kappa}),
\]
where $\tilde{C}_d$ is also continuous. We deduce \eqref{eq:boundinf2} as above, and $C_\mu$ is continuous, since we use a finite number of steps.
\end{itemize}
\end{proof}

\section{Asymptotic behavior at infinity}
In order to get the precise asymptotic behavior near infinity of the solutions $v$ of problem \eqref{eq:goveq}-\eqref{eq:inicond}, it is first necessary to show that the limit $\lim\limits_{r\rightarrow\infty}r^{\kappa}v$ exists and is finite. For this, as we mentioned in Section \ref{cvef}, we will use the function $J_\kappa$ defined in \eqref{eq:Jotakappa}. Once we know that this limit exists as a real number, we prove Propositions \ref{cs} and \ref{asymbeha} for the solutions $\theta$ of the equivalent problem \eqref{eq:newgoveq}-\eqref{eq:newinicond}, then results concerning $v$ can be extracted from $\theta$ using the transformation \eqref{eq:cv}.

\begin{proposition}\label{propoinf}
For any solution $v$ of equation \eqref{eq:goveq}, the limit $\lim\limits_{r\rightarrow\infty}r^{\kappa}v$ alway exists and is finite.
\end{proposition}
\begin{proof}
By \eqref{eq:cv} and Propositions \ref{pbehainf1} and \ref{pbehainf2}, $v(r)=O(r^{-\kappa})$ and $v'(r)=O(r^{-\kappa-1})$ near infinity. Consider the function $J_\kappa$ defined in \eqref{eq:Jotakappa}. Then by \eqref{eq:Jotanprima} and \eqref{eq:Jotakappaprima}, $J_\kappa'$ is integrable at infinity. Indeed $r^{\kappa-1-k}\abs{v'}^k=O(r^{(1-k)\kappa-2k-1})$ and $r^{\kappa-1}\abs{v}^{q-1}v=O(r^{-1-\kappa(q-1)})$. Hence $J_\kappa$ has a limit, say $L$, as $r\rightarrow\infty$. So
\[
r^\kappa v=J_\kappa(r)-r^{\kappa-k}(v')^k=J_\kappa(r)+O(r^{(1-k)\kappa-2k}).
\]
As a consequence $\lim\limits_{r\rightarrow\infty}r^\kappa v(r)=L$ and
\begin{equation}\label{eq:ele}
L=J_\kappa(r)+\int_r^\infty J_\kappa'(\tau)d\tau.
\end{equation}
In case of fast decay solutions, that is, when $L=0$, and thus $\lim\limits_{r\rightarrow\infty}J_\kappa(r)=0$ it is possible to give more precise estimates on these solutions. Indeed, since $J_\kappa(r)=-\int_r^\infty J_\kappa'(\tau)d\tau$, we obtain
\begin{equation}\label{eq:thetaboundinfty}
\abs{v(r)}\leq r^{-k}\abs{v'}^k+r^{-\kappa}\int_r^\infty\tau^{\kappa-1}(c_{n,k}^{-1}\abs{v}^q+(n+\kappa)\tau^{-k}\abs{v'}^k)d\tau.
\end{equation}
Consider $d\geq 0$ such that $v(r)=O(r^{-d})$. Using Proposition \ref{pbehainf1} and \eqref{eq:cv}, also $v'(r)=O(r^{-d-1})$. Then $v(r)=O(r^{-k(d+2)})+O(r^{-qd})$ by \eqref{eq:thetaboundinfty}. As before, we may consider a sequence letting $d_0=\kappa$ and $d_{n+1} = \min\{k(d_n+2),qd_n\}$, it follows that $(d_n)$ is nondecreasing and tends to infinity. Finally, since $q>k$, we conclude that
\[
v(r)=o(r^{-d}),\;\mbox{for any }\; d\geq 0.
\]
\end{proof}

\begin{remark}\label{depeconti}
Using Propositions \ref{pbehainf1} and \ref{pbehainf2}, and the expression of $L=L(\gamma)$ in terms of the function $J_\kappa$ given in $\eqref{eq:ele}$, we can get an analogous of \cite[Theorem 3.5]{BV06}. Hence, the function $\gamma\mapsto L(\gamma)$ is continuous on $\RR$. Furthermore, the family of functions $\{(1+r)^{\kappa}v(r,\gamma)\}_{\gamma>0,\; r\geq 0}$ is equicontinuous on $\RR$.
\end{remark}

In the following result we prove that for $k>1$ any fast decay solution necessarily has a compact support.
\begin{proposition}\label{cs}
Suppose that $k>1$. Let $\theta$ be any solution of \eqref{eq:newgoveq} such that $\lim\limits_{s\rightarrow\infty}s^{\tilde{\kappa}}\theta=0$, where $\tilde{\kappa}=\kappa/b$. Then $\theta$ has a compact support.
\end{proposition}
\begin{proof}
Suppose that $\theta$ has no compact support. We can assume that $\theta>0$ for large $s$ by Proposition \ref{zeros}. We performs the change of variables \eqref{eq:cvlog} for some $d>\tilde{\kappa}$. Since $\theta(s)=o(s^{-d}),\, \theta'(s)=o(s^{-d-1})$ near $\infty$, we get $y_d(\tau)=o(1),\,y_d'(\tau)=o(1)$ near $\infty$. Further, the function $\psi=dy_d-y_d'=-s^{d+1}\theta'$ is positive for large $\tau$ by Proposition \ref{uu'vanish}. From \eqref{eq:cveqn} and recalling that $k$ is odd, we have
\begin{equation*}
\begin{split}
&y_d''+(\eta-2d)y_d'-d(\eta-d)y_d\\
&+k^{-1}e^{((k-1)d+k+1)\tau}\psi^{1-k}(by_d'-(bd-\kappa)y_d+c_{n,k}^{-1}e^{-d(q-1)\tau}\abs{y_d}^{q-1}y_d)=0.
\end{split}
\end{equation*}
As in Proposition \ref{zeros} the maximum points $\tau$ of $y_d$ remain in a bounded set, it follows that the function $y_d$ is monotone for large $\tau$, more precisely,  $y_d'(\tau)\leq 0$ and $\lim\limits_{\tau\rightarrow\infty}e^{((k-1)d+k+1)\tau}\psi^{1-k}=\lim\limits_{s\rightarrow\infty}s^2(\theta')^{1-k}=\infty$, hence
\[
ky_d''=e^{((k-1)d+k+1)\tau}\psi^{1-k}(b\abs{y_d'}(1+o(1))+b(d-\tilde{\kappa})y_d(1+o(1))).
\]
Since $d-\tilde{\kappa}>0$, there exists $C>0$ such that $y_d''\geq Ce^{((k-1)d+k+1)\tau}\psi^{2-k}$ for large $\tau$. Using this we arrive at the inequality
\[
-\psi'=y_d''+d\abs{y_d'}\geq Ce^{((k-1)d+k+1)\tau}\psi^{2-k},
\]
which implies that the function $\psi^{k-1}+Ce^{((k-1)d+k+1)\tau}/\left(d+\frac{k+1}{k-1}\right)$ is nonincreasing, which is impossible.
\end{proof}

The following result gives a precise asymptotic expansion of the slow decay solutions $\theta$ of equation \eqref{eq:newgoveq}, improving the results in Proposition \ref{uu'vanish}. Using the transformation \eqref{eq:cv} we transfer this information to the function $v$, thus obtaining the corresponding result for slow decay solutions $v$ of equation \eqref{eq:goveq}.

\begin{proposition}\label{asymbeha}
Let $\theta$ be any solution of \eqref{eq:newgoveq} such that $\tilde{L}=\lim\limits_{s\rightarrow\infty}s^{\tilde{\kappa}}\theta>0$, where $\tilde{\kappa}=\frac{\kappa}{b}$. Then
\begin{equation}\label{eq:derivainfty}
\lim\limits_{s\rightarrow\infty}s^{\tilde{\kappa}+1}\theta'=-{\tilde{\kappa}}\tilde{L},
\end{equation}
and 
\begin{equation}\label{thetaasym}
\theta(s)=\begin{cases}
s^{-\tilde{\kappa}}(\tilde{L}+(\tilde{A}+o(1))s^{-\tilde{\nu}}),&\mbox{if } (q-k)\tilde{\kappa}>k+1,\\
s^{-\tilde{\kappa}}(\tilde{L}+(\tilde{A}+\tilde{B}+o(1))s^{-\tilde{\kappa}(q-1)}),&\mbox{if } (q-k)\tilde{\kappa}=k+1,\\
s^{-\tilde{\kappa}}(\tilde{L}+(\tilde{B}+o(1))s^{-\tilde{\kappa} (q-1)}),&\mbox{if } (q-k)\tilde{\kappa}<k+1,
\end{cases}
\end{equation}
where 
\[
\tilde{\nu}=(k-1)\tilde{\kappa}+k+1,\;\tilde{A}=\frac{(k\tilde{\kappa}-(\tilde{n}-(k+1)))(\tilde{\kappa}\tilde{L})^k}{b\tilde{\nu}},\; \tilde{B}=\frac{\tilde{L}^q}{bc_{n,k}\tilde{\kappa}(q-1)}.
\]
Moreover, differentiating term to term gives an expansion of $\theta'$.
\end{proposition}
\begin{proof}
We performs the transformation \eqref{eq:cvlog} with $d=\tilde{\kappa}$, so that $\theta(s)=s^{-\tilde{\kappa}}y_{\tilde{\kappa}}(\tau)$. For large $s$, we have that $\theta'(s)=s^{-(\tilde{\kappa}+1)}(y_{\tilde{\kappa}}'(\tau)-\tilde{\kappa} y_{\tilde{\kappa}}(\tau))<0$. Then $\tilde{\kappa} y_{\tilde{\kappa}}-y_{\tilde{\kappa}}'>0$ for large $\tau$. Thus, taking into account \eqref{eq:cvsys}, the system \eqref{eq:system} takes the form:
\begin{equation}\label{eq:system2}
\begin{split}
&y_{\tilde{\kappa}}'=\tilde{\kappa} y_{\tilde{\kappa}}-Y_{\tilde{\kappa}}^{\frac{1}{k}},\\
&Y_{\tilde{\kappa}}'=k({\tilde{\kappa}}-\eta)Y_{\tilde{\kappa}}+e^{\tilde{\nu}\tau}(b({\tilde{\kappa}}y_{\tilde{\kappa}}-Y_{\tilde{\kappa}}^{\frac{1}{k}})+c_{n,k}^{-1}e^{-{\tilde{\kappa}}(q-1)\tau}y_{\tilde{\kappa}}^{q}).
\end{split}
\end{equation}
Obviously the function $y_{\tilde{\kappa}}$ converges to $\tilde{L}$, and $y_{\tilde{\kappa}}'$ is bounded near infinity, because $\theta'=O(s^{-(\tilde{\kappa}+1)})$ near infinity, whence $Y_{\tilde{\kappa}}$ is bounded. Therefore, we have only two alternatives, either $Y_{\tilde{\kappa}}$ is monotone for large $\tau$, thus having a finite limit, say $\lambda$, and consequently, $y_{\tilde{\kappa}}'$ converges to $\tilde{\kappa}\tilde{L}-\lambda^{\frac{1}{k}}$ being $\lambda=(\tilde{\kappa}\tilde{L})^k$, or else for large $\tau$, the critical points of $Y_{\tilde{\kappa}}$ form an increasing sequence $(\tau_m)$ tending to $\infty$. Then 
\[
bY^{\frac{1}{k}}_{\tilde{\kappa}}(\tau_m)=b{\tilde{\kappa}} y_{{\tilde{\kappa}}}(\tau_m)+c_{n,k}^{-1}e^{-\tilde{\kappa}(q-1)\tau_m}y_{\tilde{\kappa}}^q(\tau_m)+k(\tilde{\kappa}-\eta)e^{-\tilde{\nu}\tau_m}Y_{\tilde{\kappa}}(\tau_m).
\]
Hence $\lim\limits_{m\rightarrow\infty} Y_{\tilde{\kappa}}(\tau_m)=(\tilde{\kappa}\tilde{L})^k$. In either case $\lim\limits_{\tau\rightarrow\infty} Y_{\tilde{\kappa}}(\tau)=(\tilde{\kappa}\tilde{L})^k$, which is equivalent to \eqref{eq:derivainfty} and implies $\lim\limits_{\tau\rightarrow\infty}y'_{\tilde{\kappa}}(\tau)=0$. It is remain to consider $Y'_{\tilde{\kappa}}$. Either it is monotone for large $\tau$, thus $\lim\limits_{\tau\rightarrow\infty} Y_{\tilde{\kappa}}'(\tau)=0$; or for large $\tau$, the critical points of $Y_{\tilde{\kappa}}'$ form an increasing sequence $(s_m)$ tending to $\infty$. Using the fact that $Y_{\tilde{\kappa}}''(s_m)=0$, and by computation, at the point $\tau=s_m$, we obtain
\begin{eqnarray*}
&&\left(\frac{2}{k+1}Y_{\tilde{\kappa}}^{\frac{1-k}{k}}-k(\tilde{\kappa}-\eta)e^{-\tilde{\nu}\tau}\right)Y_{\tilde{\kappa}}'\\
&=&(k(\kappa+2)+qc_{n,k}^{-1}e^{-\tilde{\kappa}(q-1)\tau}y_{\tilde{\kappa}}^{q-1})y_{\tilde{\kappa}}'+c_{n,k}^{-1}(\tilde{\nu}-\tilde{\kappa}(q-1))e^{-\tilde{\kappa}(q-1)\tau}y_{\tilde{\kappa}}^{q}
\end{eqnarray*}
from which $\lim\limits_{m\rightarrow\infty} Y_{\tilde{\kappa}}'(s_m)=0$. In either case, $\lim\limits_{\tau\rightarrow\infty} Y_{\tilde{\kappa}}'(\tau)=0$. It follows from \eqref{eq:system2} that
\begin{eqnarray*}
y_{\tilde{\kappa}}'&=&-(bc_{n,k})^{-1}e^{-\tilde{\kappa}(q-1)\tau}y_{\tilde{\kappa}}^q-b^{-1}e^{-\tilde{\nu}\tau}(k(\tilde{\kappa}-\eta)Y_{\tilde{\kappa}}-Y'_{\tilde{\kappa}})\\
&=&-\left(\frac{\tilde{L}^q}{bc_{n,k}}+o(1)\right)e^{-\tilde{\kappa}(q-1)\tau}-\tilde{\nu}(A+o(1))e^{-\tilde{\nu}\tau}.
\end{eqnarray*}
Hence $y'_{\tilde{\kappa}}=-\tilde{\nu}(A+o(1))e^{-\tilde{\nu}\tau}$ if $\tilde{\kappa}(q-1)>\tilde{\nu}$, or equivalently $(q-k)\tilde{\kappa}>k+1$; and $y'_{\tilde{\kappa}}=-\left(\tilde{\nu} A+\frac{\tilde{L}^q}{bc_{n,k}}+o(1)\right)e^{-\tilde{\nu}\tau}$ if $\tilde{\kappa} (q-1)=\tilde{\nu}$; and $y'_{\tilde{\kappa}}=-\left(\frac{\tilde{L}^q}{bc_{n,k}}+o(1)\right)e^{-\tilde{\kappa}(q-1)\tau}$ if $\tilde{\kappa}(q-1)<\tilde{\nu}$. An integration easily leads to estimates in \eqref{thetaasym}. This gives also an expansion of the derivatives, by computing $\theta'=-s^{-(\tilde{\kappa}+1)}(\tilde{\kappa} y_{\tilde{\kappa}}-y_{\tilde{\kappa}}')$:
\begin{equation*}
\theta'(s)=\begin{cases}
-s^{-(\tilde{\kappa}+1)}(\tilde{\kappa}\tilde{L}+(\tilde{\kappa}+\tilde{\nu})(\tilde{A}+o(1))s^{-\tilde{\nu}}),&\mbox{if } (q-k)\tilde{\kappa}>k+1,\\
-s^{-(\tilde{\kappa}+1)}(\tilde{\kappa}\tilde{L}+(\tilde{\kappa}+\tilde{\nu})(\tilde{A}+\tilde{B}+o(1))s^{-\tilde{\nu}}),&\mbox{if } (q-k)\tilde{\kappa}=k+1,\\
-s^{-(\tilde{\kappa}+1)}(\tilde{\kappa}\tilde{L}+\tilde{\kappa} q(\tilde{B}+o(1))s^{-\tilde{\kappa} (q-1)}),&\mbox{if } (q-k)\tilde{\kappa}<k+1,
\end{cases}
\end{equation*}
which correspond to a derivation term to term.
\end{proof}

Using \eqref{eq:cv} the preceding proposition translate for $v$ as
\begin{proposition}\label{asymbehav}
Let $v$ be any solution of \eqref{eq:goveq} such that $L=\lim\limits_{r\rightarrow\infty}r^{\kappa}v>0$. Then
\[
\lim\limits_{r\rightarrow\infty}r^{\kappa+1}v'=-\kappa L,
\]
and 
\begin{equation}\label{vasym}
v(r)=\begin{cases}
r^{-\kappa}(L+(A+o(1))\left(\frac{2k}{k+1}\right)^{\frac{(k+1)(\kappa+2)}{2}}r^{-\nu}),&\mbox{if } (q-k)\kappa>2k,\\
r^{-\kappa}(L+(A+B+o(1))\left(\frac{2k}{k+1}\right)^{\frac{(k+1)\kappa q}{2k}}r^{-\kappa (q-1)}),&\mbox{if } (q-k)\kappa=2k,\\
r^{-\kappa}(L+(B+o(1))\left(\frac{2k}{k+1}\right)^{\frac{(k+1)\kappa q}{2k}}r^{-\kappa (q-1)}),&\mbox{if } (q-k)\kappa<2k,
\end{cases}
\end{equation}
where 
\[
\nu=(k-1)\kappa+2k,\;A=\frac{(k\kappa-(n-2k))\left(\frac{k+1}{2k}\right)^{\frac{(k+1)(\kappa+2)}{2}}(\kappa L)^k}{\nu},\; B=\frac{\left(\frac{k+1}{2k}\right)^{\frac{(k+1)\kappa q}{2k}}L^q}{\kappa (q-1)c_{n,k}}.
\]
Moreover, an asymptotic expansion of $v'$ is also provided by Proposition \ref{asymbeha}, since $v'(r)=r^{b-1}\theta'(s)$. 
\end{proposition}

\begin{remark} We see from Proposition \ref{asymbehav} that the first two terms in the asymptotic expansion of the slow decay solutions $v(r)$ in \eqref{vasym}, when $q>k+(2k)/\kappa$, are not influenced by the nonlinear term, but that are when $q\leq k+(2k)/\kappa$, see Figure 1 below. Notice that in the semilinear case when $k=1$, we recover the result in \cite[Theorem 1, part (ii)]{PeTeWe86}. 
 \end{remark}
 
\begin{figure}[H]\label{Fig1}
\centering
\begin{tikzpicture}
\draw[->] (0,0) --( 4,0) node[right]{\small $\kappa$};
\draw[->] (0,0) --( 0,4) node[above]{\small $q$};		
\draw[] (0,.45) -- ( 3.9,.45);
\draw (1.9,3.6) node {$q(\kappa)=k+(2k)/\kappa$};
\draw (-0.4,.8) node[anchor=north west] {\small $k$};
\draw (2.2,1.9) node {$\sim E$};
\draw (1,.7) node {$E$};
\draw (0.1,3.9) .. controls (.5,.5) and (3,.5) .. (3.9,.5);
\draw[->] (1,.9) -- ++(45:.48);
\draw[->] (1,.9) -- ++(135:.8);
\draw[->] (2.5,2.1) -- ++(45:.8);	
\end{tikzpicture}
\caption{The figure shows the exponent $q$ as a function of the parameter $\kappa$. Also shows the localization of the regions $E$ and $\sim E$ on the quadrant $(\kappa,q)$, respectively, with or without effect of the nonlinear source on the first two terms in the asymptotic representation of the profiles.}
\end{figure}
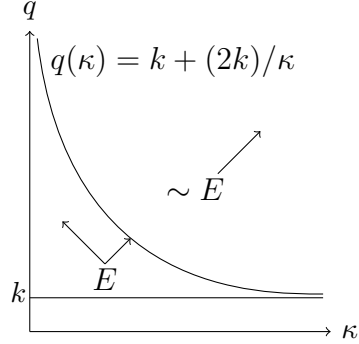

\section{Existence of nonnegative solutions}
Here we study the existence of nonnegative solutions $\theta$ of equation \eqref{eq:newgoveq}. When $\kappa<n$, we first prove the existence of slow decay solutions for $\gamma$ small enough. Recall that all the information about $\theta$ is transfer via \eqref{eq:cv} to the solution $v$ of equation \eqref{eq:goveq}.
\begin{proposition}\label{slowdecay}
Assume $\kappa<n$. Let $\underline{\gamma}>0$ be defined as in Proposition \ref{zeros}. Then for any $\gamma\in (0,\underline{\gamma}],\, \theta(s,\gamma)>0$ on $[0,\infty)$ and $\tilde{L}(\gamma)>0$.
\end{proposition}
\begin{proof}
Let $\gamma\in (0,\underline{\gamma}]$. By construction of $\underline{\gamma},\; \theta=\theta(s,\gamma)>0$ from Proposition \ref{zeros}. If $\tilde{L}(\gamma)=0$, then $\theta$ has a compact support from Proposition \ref{cs}, which is impossible.
\end{proof}
Next we consider the subcritical case $k<q<q^*(k)$ and prove the existence of fast solutions. 
\begin{theorem}\label{compa_supp}
Assume $\kappa<n$ and $k<q<q^*(k)$. Then there exists $\gamma>0$ such that $\theta(\cdot,\gamma)$ is nonnegative and such that $\tilde{L}(\gamma)=0$. If $k>1$, it has a compact support.
\end{theorem}
\begin{proof}
Let 
\[
A=\{\gamma>0: \theta(\cdot,\gamma)>0\; \mbox{on}\; (0,\infty)\; \mbox{and}\; \tilde{L}(\gamma)>0\},
\]
\[
B=\{\gamma>0: \theta(\cdot,\gamma)\; \mbox{has at least an isolated zero}\}.
\]
From Propositions \ref{slowdecay} and \ref{vzeros} (here we recall that $\theta(s,\gamma)=v(r,\gamma)$, where $r=(bs)^{\frac{1}{b}}$), $A$ and $B$ are nonempty since: $A\supset (0,\underline{\gamma}]$ and $B\supset [\overline{\gamma},\infty)$. The rest of the proof is analogous to that of Theorem 3.9 of \cite{BV06} whose proof include the Theorem 3.5 (see Remark \ref{depeconti}) of the same paper. Therefore, there exists $\gamma>0$ such that $\theta(\cdot,\gamma)$ is nonnegative and $\tilde{L}(\gamma)=0$. From Proposition \ref{cs}, $\theta$ has a compact support.
\end{proof}
\begin{remark}
In Proposition \ref{slowdecay} it was shown that if $\kappa<n$, and $q>k$ arbitrary,
any solution of problem \eqref{eq:newgoveq}-\eqref{eq:newinicond} is positive and a slow solution provided that
\[
0<\gamma<\underline{\gamma}=(c_{n,k}(n-\kappa))^{1/(q-1)}.
\]
It is interesting to note that if $\kappa=\kappa_0\in (n/2, n)$, then $\theta(0)>\underline{\gamma}$.
\end{remark}

In the supercritical case $q\geq q^{\ast}(k)$ we give sufficient conditions assuring that all the solutions are positive, and necessarily slow solutions.
\begin{proposition}\label{super_crit}
Assume that $\kappa\leq\frac{n}{2}$ and $q\geq q^{\ast}(k)$. Then for any $\gamma>0,\,\theta(s,\gamma)>0$ on $[0,\infty)$ and $\tilde{L}=\tilde{L}(\gamma)>0$.
\end{proposition}
\begin{proof}
We use the function $\mathcal{V}=\mathcal{V}_{\lambda,\sigma,\mu}$ defined in \eqref{eq:VPoho}, where $\lambda>0,\sigma$ and $\mu$ will be chosen later. It is continuous at 0 and $\mathcal{V}_{\lambda,\sigma,\mu}(0)=0$, from \eqref{thetaprimezero}. Suppose that $\theta(s_0)=0$ for some first real $s_0>0$. Then $\mathcal{V}_{\lambda,\sigma,\mu}(s_0)=s^{\lambda}k(\theta'(s_0))^{k+1}/(k+1)\geq 0$. Suppose that for some $\lambda,\sigma,\mu$, the five terms giving $\mathcal{V}'$ are nonpositive. Then $\mathcal{V}\equiv \mathcal{V}'\equiv 0$ on $[0,s_0]$. Hence $s\theta'+(b\sigma-\mu+\kappa)/(2b)\theta\equiv 0$, that is, $s^{\frac{b\sigma-\mu+\kappa}{2b}}\theta$ is constant. Thus, $\theta\equiv 0$ if $b\sigma-\mu+\kappa\neq 0$, or $\theta\equiv\gamma$ if $b\sigma-\mu+\kappa=0$. This is impossible since $\theta(0)\neq\theta(s_0)$. Take $\lambda=\delta+1,\, \sigma=\frac{n-2k}{2k}$ and $\mu=b\sigma+\kappa-n$. Thus
\[
\mathcal{V}(s)=s^{\delta+1}\left(\frac{k}{k+1}(\theta')^{k+1}+\frac{c_{n,k}^{-1}}{q+1}\abs{\theta}^{q+1}+\left(\frac{n-2k}{k+1}+\kappa-n\right)\frac{\theta^2}{2}+\frac{n-2k}{2k} s^{-1}(\theta')^k\theta\right).
\]
After computation we find
\begin{eqnarray*}
s^{-\delta}\mathcal{V}'(s)&=&-(2kc_{n,k})^{-1}\left(n-2k-\frac{n(k+1)}{q+1}\right)
\abs{\theta}^{q+1}-\frac{n(k-1)+4k}{4(k+1)b}\left(n-2\kappa\right)\theta^2\\
&-&b\left(s\theta'+\frac{n}{2b}\theta\right)^2
\end{eqnarray*}
and all the terms are nonpositive from our assumptions, thus $\theta>0$ on $[0,\infty)$. Moreover, $\tilde{L}(\gamma)>0$, since if it were not true, $\theta$ would have a compact support by Proposition \ref{cs}.
\end{proof}
  
We close this section applying to equation \eqref{eq:rad_kappa_0} the previous results with $\kappa=\kappa_0=\frac{\alpha_0}{\beta_0}=\frac{2k}{q-k}$, and show our main result concerning type I solutions.
 
\begin{proof}[Proof of Theorem \ref{beha_type_I_sol}] One has $\kappa_0>0$ since $q>k$.
\begin{itemize}
\item [$(a)$] A rigorous proof of the existence of $v$ can be obtained by fixed point techniques, as was mentioned in Section \ref{cvef}, while its behavior follows from Proposition \ref{propoinf}.

\item [$(b)$] Condition $q_c(k)<q$ is equivalent to $\kappa_0<n$, and Proposition \ref{slowdecay} applies.

\item [$(c)$] If $q_c(k)<q<q^\ast(k)$, then Theorem \ref{compa_supp} shows the existence of fast nonnegative decay solutions $\theta$ and hence solutions $v$ with the same properties. In fact, notice that $L(\gamma)=b^{\tilde{\kappa}}\tilde{L}(\gamma)=0$. If $k=1$, the second part of the proof follows from statements $(i),\,(ii)$ and $(iii)$ in Theorem 1 of \cite{HaWe82}. Now for any $p\geq 1$, there exists $C>0$ such that for any $t>0$,
\begin{equation}\label{eq:Elep_norm}
\norm{u(t)}_p=Ct^{\frac{\frac{n}{p\kappa_0}-1}{q-1}}\norm{v}_p.
\end{equation}
If $k>1,\, v$ has a compact support by Proposition \ref{cs} and $u(t)\in L^p(\mathbb{R}^n)$. Moreover $\lim_{t\rightarrow 0}\norm{u(t)}_p=0$ whenever $p<\frac{n}{\kappa_0}$, from \eqref{eq:Elep_norm}. For fixed $\epsilon>0$, by Proposition \ref{cs}, $\theta$ then $v$ has a compact support and $\sup_{\abs{x}\geq\epsilon}\abs{u(t,x)}=0$ for $t\leq t(\epsilon)$ small enough. Hence $\lim_{t\rightarrow 0}\sup_{\abs{x}\geq\epsilon}\abs{u(t,x)}=0$ for any $\epsilon>0$.

\item [$(d)$] Here we apply Proposition \ref{super_crit}. Indeed if $q\geq q^*(k)$, then $\kappa_0\leq\frac{n-2k}{k+1}<\frac{n}{2}$.
 
\item [$(e)$] If $k<q\leq q_c(k)$, then $n\leq\kappa_0$. Hence all the nontrivial solutions $v$ change sign, from Proposition \ref{zeros}, part $(ii)$. The existence of an infinity of fast solutions (and hence compactly supported) $v$ follows from \cite[Theorem 3.6, part $ii)$]{BV06}, whose proof can easily be adapted to the present situation.
\end{itemize}
\end{proof}

\section{An explicit profile}
In this final section we give an example of the application of our general results in the previous sections. We construct a family of explicit solutions of Problem \eqref{eq:goveq}-\eqref{eq:inicond} when $q$ is supercritical and $\kappa=\frac{4}{q-1}$, but only in the case $k=1$. The special form of these solutions given by $v(r)=A(1+Br^C)^{-D}$, with $A,B,C$ and $D$ positive constants allow to balance the linear and nonlinear terms in equation \eqref{eq:goveq}, which seems impossible when $k\neq 1$. Thus, it is a very fortunate fact to be able to obtain exact solutions, at least in the semilinear case.

\begin{proposition}\label{exactsol}
Suppose $q>q^*(1)=(n+2)/(n-2)$.
Then the function
\[
v(r)=A(q,n)(1+B(q,n)\,r^2)^{-\frac{2}{q-1}}
\]
in which 
\[
A(q,n)=\left(\frac{\kappa(\kappa+2)}{n-2-\kappa}\right)^{\frac{1}{q-1}}
\]
\[
B(q,n)=(n-2-\kappa)^{-1}
\]
is an exact solution of Problem \eqref{eq:goveq}-\eqref{eq:inicond} with $k=1$ if
\[
\kappa=\frac{4}{q-1}\;\;\ \mbox{and}\;\;\; \gamma=A(q,n).
\]
\end{proposition}
The proof of Proposition \ref{exactsol} consists of a lengthy but elementary computation. We shall omit it. 
It is interesting to recall here that if $q=q^*(k)$ with $k$ odd, the equation
\[
S_k(D^2v)+\abs{v}^{q-1}v=0
\]
also has a family of explicit radial solutions defined on $\RR^n$ \cite{ClMM98}. It is given by 
\[
v(r)=\left[\frac{\{\binom{n}{k}\left(\frac{n-2k}{k}\mu\right)^k\}^{\frac{1}{k+1}}}{1+\mu r^2}\right]^{\frac{n-2k}{2k}},\; \mu>0.
\]
Note that, whereas $q$ is completely determined in this case, $v(0)$ can be chosen entirely arbitrarily by adjusting the parameter $\mu$. 

Observe that if $\kappa=\frac{4}{q-1}$, then
\[
L=\lim_{r\rightarrow\infty}r^\kappa v(r)=(\kappa(\kappa+2)(n-2-\kappa))^{\frac{1}{q-1}}>0.
\]
Therefore, the family of explicit solutions $v$ given in Proposition \ref{exactsol} consists of slow decay solutions. Their asymptotic behavior is given by Proposition \ref{asymbehav}, where the first case in \eqref{vasym} $(q-1)\kappa=4>2$ applies.
\medskip

\section*{\bf Acknowledgment} 

\noindent  The author is supported by ANID Fondecyt Grant Number 1221928, Chile

\bibliographystyle{plain}
\bibliographystyle{apalike}
\bibliography{kHessianbib}
\end{document}